\documentclass[12pt]{amsart}

\usepackage[utf8]{inputenc}						
\usepackage{amsmath,amsfonts,amssymb,amsthm}	

\usepackage{geometry}		
\usepackage{mathtools}		
\usepackage{hyperref}		
\usepackage{setspace}		
\usepackage{cite}			
\usepackage{color}			

\definecolor{midpurple}{rgb}{0.6,0.2,0.4}
\hypersetup{colorlinks=true,citecolor=blue,linkcolor=blue,urlcolor=midpurple}

\numberwithin{equation}{section}			
\newtheorem{theorem}{Theorem}[section]			
\newtheorem{proposition}[theorem]{Proposition}	
\newtheorem{lemma}[theorem]{Lemma}

\newtheorem*{theorem*}{Theorem}
\newtheorem*{proposition*}{Proposition}
\newtheorem*{lemma*}{Lemma}
\newtheorem*{corollary*}{Corollary}
\newtheorem*{fact*}{Fact}

\newtheorem*{conjecture*}{Conjecture}
\newtheorem*{question*}{Question}
\newtheorem*{remark*}{Remark}
\newtheorem*{definition*}{Definition}

\renewcommand{\leq}{\leqslant}	
\renewcommand{\geq}{\geqslant}	

\newcommand{\N}{\mathbb{N}}		
\newcommand{\Z}{\mathbb{Z}}		
\newcommand{\Q}{\mathbb{Q}}		
\newcommand{\R}{\mathbb{R}}		
\newcommand{\C}{\mathbb{C}}		
\newcommand{\T}{\mathbb{T}}		

\newcommand{\eps}{\varepsilon}			

\newcommand{\wt}{\widetilde}		
\newcommand{\wh}{\widehat}			
\DeclareMathOperator{\Supp}{Supp}	

\newcommand{\transp}{\mathsf{T}}						

\newcommand{\dm}{\mathrm{d}m}

\newcommand{\dalpha}{\mathrm{d}\alpha}

\newcommand{\dlambda}{\mathrm{d}\lambda}

\newcommand{\bfb}{\mathbf{b}}

\newcommand{\bfk}{\mathbf{k}}

\newcommand{\bfm}{\mathbf{m}}
\newcommand{\bfn}{\mathbf{n}}
\newcommand{\bfr}{\mathbf{r}}

\newcommand{\bfu}{\mathbf{u}}

\newcommand{\bfx}{\mathbf{x}}

\newcommand{\bfalpha}{\boldsymbol{\alpha}}	
	
\newcommand{\bfell}{\boldsymbol{\ell}}	
\newcommand{\bfgamma}{\boldsymbol{\gamma}}

\newcommand{\bftheta}{\boldsymbol{\theta}}	

\newcommand{\bfxi}{\boldsymbol{\xi}}	

\newcommand{\dbfx}{\mathrm{d}\mathbf{x}}
\newcommand{\dbfalpha}{\mathrm{d}\boldsymbol{\alpha}}
	
\newcommand{\dbftheta}{\mathrm{d}\boldsymbol{\theta}}
\newcommand{\dbfxi}{\mathrm{d}\boldsymbol{\xi}}

\newcommand{\frakM}{\mathfrak{M}}		
\newcommand{\frakm}{\mathfrak{m}}		
\newcommand{\calU}{\mathcal{U}}			

\begin{document}

\title{On restriction estimates for discrete quadratic surfaces}

\author{Kevin Henriot, Kevin Hughes}

\date{}

\begin{abstract}
We obtain truncated restriction estimates of an unexpected form for
discrete surfaces
\begin{align*}
	S = \{\, ( n_1 , \dots , n_d , R( n_1 , \dots, n_d ) ) \,,\, n_i \in [-N,N] \cap \Z \,\},
\end{align*}
where $R$ is an indefinite quadratic form with integer matrix.
\end{abstract}

\maketitle

\section{Introduction}
\label{sec:intro}

We fix a non-degenerate
quadratic form $R$ in $d$ variables with integer matrix.
We are interested in restriction estimates
for quadratic surfaces in $\Z^{d+1}$ of the form
\begin{align}
\label{eq:intro:QuadSurface}
	S = \{\, ( n_1 , \dots , n_d , R( n_1 , \dots, n_d ) ) \,,\, n_i \in [-N,N] \cap \Z \,\},
\end{align}
in the case where $R$ is indefinite.
This paper should be seen as a companion to~\cite{HH:epsremoval_parab},
which concerned the case $R(n) = n^k$ of $k$-th powers and
the case $R(\bfn) = n_1^d + \dotsb + n_d^k$ of `$k$-paraboloids'; 
the methods employed here are similar but our results take a different shape.

In the case $d = 1$, $R(x) = x^2$ of the $2D$ parabola, Bourgain~\cite{Bourgain:ParabI}
resolved the natural restriction conjecture in the supercritical range,
via discrete versions of the Tomas--Stein argument~\cite[Chapter~7]{Wolff:Book}
and the Hardy--Littlewood circle method.
By powerful new methods of multilinear harmonic analysis,
Bourgain and Demeter~\cite[Theorem~2.4]{BD:DecouplConj}
later established the natural restriction conjecture 
for arbitrary definite irrational paraboloids 
$R(\bfx) = \theta_1 x_1^2 + \dotsb + \theta_d x_d^2$,
$\theta_i > 0$, up to $\eps$ losses.
In a subsequent work~\cite[Corollary~1.3]{BD:Hypersurf},
they also obtained the conjectured estimate for indefinite paraboloids.
To state those results precisely, we set up some notation.
The extension operator
acting on a sequence $a : \Z^d \rightarrow \C$ supported on $[-N,N]^d$ 
is denoted by
\begin{align}
\label{eq:intro:FaDef}
	F_a( \alpha, \bftheta )
	&= \sum_{\bfn \in \Z^d} a(\bfn) e( \alpha R(\bfn) + \bftheta \cdot \bfn )
	&&(\alpha \in \T, \bftheta \in \T^d).
\end{align}

\begin{theorem}[Bourgain--Demeter~\cite{BD:Hypersurf}, special case]
\label{thm:intro:RestrHyperbParab}
Suppose that $R$ is a non-degenerate indefinite quadratic form in $d$ variables
with integer matrix and signature $(p,q)$, and let $s = \min(p,q)$.
We have
\begin{align*}
	\| F_a \|_p^p
	\lesssim_\eps
	\begin{cases}
		N^{\frac{sp}{2} - s + \eps} \| a \|_2^p 
		& \text{for $2 \leq p \leq \frac{2(d - s + 2)}{d - s}$,}
		\\
		N^{\frac{dp}{2} - (d+2)} \| a \|_2^p
		&\text{for $p > \frac{2(d - s + 2)}{d - s}$.}
	\end{cases}
\end{align*}
\end{theorem}

While this is stated only for diagonal forms in~\cite{BD:Hypersurf},
a simple diagonalization argument allows one to reduce to this case.
There is also an extra $N^\eps$ factor in the supercritical range in that reference, 
which can be removed by (a minor variant of) Bourgain's $\eps$-removal estimate~\cite{Bourgain:ParabI};
we refer to Appendix~\ref{sec:appdiag} for the details.
The exponent of $N$ in Theorem~\ref{thm:intro:RestrHyperbParab} 
is sharp for even integer exponents $p > \frac{2(d - s + 2)}{d - s}$,
as can be seen by taking $a \equiv 1$ and using the circle method to obtain an asymptotic.
As explained in~\cite{BD:Hypersurf}, the lower bound 
$\| F_a \|_p^p \gtrsim N^{\frac{sp}{2} - s} \| a \|_2^p$ also holds for
a sequence supported on a subspace of dimension $s$.
More precisely, assume for simplicity that
$R(\bfx) = \sum_{i=1}^s x_i^2 - \sum_{i=s+1}^d x_i^2$ with $s \leq d/2$,
then $a(\bfn) = 1_{[N]^{2s}}(\bfn) \prod_{i=1}^s 1_{n_i = n_{s+i}}$ is an (approximate) extremizer.
Note that $|F_a| \leq N^{s/2} \| a \|_2 \leq N^{d/4} \| a \|_2$ is rather small in that case.
Our main result adapts the proof of Bourgain's $\eps$-removal lemma~\cite{Bourgain:ParabI}
to indefinite quadratic parabolas, and we obtain an intriguing bound
for the truncated integral.

\begin{theorem}
\label{thm:intro:TruncRestrQuad}
Let $R$ be a non-degenerate indefinite quadratic
form in $d \geq 1$ variables with integer matrix and signature $(p,q)$,
and let $s = \min(p,q)$.
There exists $C > 0$ such that
\begin{align}
\label{eq:intro:TruncRestrQuad}
	\int_{ |F_a| \geq CN^{d/4} \| a \|_2 } |F_a|^p \dm
	\lesssim
	N^{\frac{dp}{2} - (d+2)} \| a \|_2^p
\end{align}
for $p > \frac{2(d+2)}{d}$.
\end{theorem}

Note that the upper bound above is of order less
than the order $N^{\frac{sp}{2} - s + O(\eps)} \| a \|_2^p$
of the complete integral, as given by Theorem~\ref{thm:intro:RestrHyperbParab}.
This can be seen as an inverse result
saying that for sequences $a : \Z^d \rightarrow \C$ 
maximizing the ratio $\| F_a \|_p / \| a \|_2$,
the ``mass'' of the integral $\int |F_a|^p \dm$ 
is concentrated on a set where $|F_a|$ has square-root cancellation
(in comparison with the trivial Cauchy-Schwarz bound $|F_a| \leq N^{d/2} \| a \|_2$).
This is consistent with the above example of maximizer
supported on the subvariety $n_i = n_{s+i}$, $1 \leq i \leq s$.
Such a behavior would be impossible in the definite case,
where the tail integral over $\{ |F_a| \leq N^{d/2 - \zeta} \| a \|_2 \}$ always contributes
less than the main term, for any $\zeta > 0$.

In proving Theorem~\ref{thm:intro:TruncRestrQuad},
we do not have a simple diagonalization argument at our disposal
to estimate the truncated integral~\eqref{eq:intro:TruncRestrQuad}.
Therefore we adapt the approach of Bourgain~\cite{Bourgain:ParabI}
for the parabola $(x_1,\dots,x_d,x_1^2 + \dotsb + x_d^2)$
to use multidimensional exponential sum estimates,
whereas in the diagonal case the relevant exponential sum~\eqref{eq:prelims:FDef} 
splits into one-dimensional quadratic Weyl sums.
This process is successful since efficient bounds on quadratic exponential sums
are known classically, and we do not encounter 
certain difficulties described in~\cite{HH:epsremoval_parab}
for surfaces of high degree.

We note finally that the related problem of obtaining $\eps$-free estimates
in the full supercritical range for
indefinite irrational quadratic forms $R$ is still open,
although there is partial progress in this direction by Godet and Tzvetkov~\cite{GT:IndefParab} 
and Wang~\cite{Wang:IndefParab}.
In the definite case, the question has been settled
recently by Killip and Vi\c{s}an~\cite{KV:Parab}.


\textbf{Acknowledgements.}
The first author would like to thank Akos Magyar for stimulating discussions
on exponential sums in many variables.
The second author thanks Hiro Oh, Yuzhao Wang and Trevor Wooley for stimulating discussions on restriction theory.

\section{Notation}
\label{sec:notat}

For functions $f : \T^d \rightarrow \C$
and $g : \Z^d \rightarrow \C$ we define
the Fourier transforms of $f$ and $g$ by
$\wh{f}( \bfk ) = \int_{\T^d} f( \bfalpha ) e( - \bfalpha \cdot \bfk ) \dbfalpha$
and $\wh{g}( \bfalpha ) = \sum_{\bfn \in \Z^d} g(\bfn) e( \bfalpha \cdot \bfn ) $.
For a function $h : \R^d \rightarrow \C$
we define the Fourier transform by 
$\wh{h}(\bfxi) = \int_{\R^d} f(\bfx) e( - \bfxi \cdot \bfx ) \dbfx$.
Given a function $f : \R^d \rightarrow \R$
and two subsets $A,B$ of $\R^d$,
we write $A \prec f \prec B$
when $0 \leq f \leq 1$ everywhere,
$f = 1$ on $A$ and $f = 0$ outside $B$.

When $\mathcal{P}$ is a certain property, we
let $1_\mathcal{P}$ denote the boolean equal to $1$
when $\mathcal{P}$ holds and $0$ otherwise,
and when $E$ is a set we define the indicator function 
of $E$ by $1_E(x) = 1_{x \in E}$.
When $p \in [1,+\infty]$ is an exponent, 
we systematically denote by $p' \in [1,+\infty]$ 
its dual exponent satisfying $\frac{1}{p} + \frac{1}{p'} = 1$.
We let $\dm$ denote the Lebesgue measure on $\R^d$, or on 
$\T^d$ identified with any fundamental domain of the form $[\theta,1 + \theta)^d$.
For $q \geq 2$ we occasionally use $\Z_q$ as a shorthand for the group $\Z/q\Z$.

Throughout the article, we use the letter $\eps$ 
generically to denote a constant
which can be taken arbitrarily small, and whose value
may change in each occurence.

\section{Arc mollifiers}
\label{sec:mollif}

This section is a specialization of~\cite[Section~6]{HH:epsremoval_parab}
to the quadratic case $k=2$, and we include it for completeness.
Its aim is to describe a technical tool
due to Bourgain~\cite[Section~3]{Bourgain:ParabI}
and used in the proof of 
Theorem~\ref{thm:intro:TruncRestrQuad},
which consists essentially in a partition of unity adapted to major arcs.

We fix an integer $N \geq 1$, to be thought of as large.
We fix a smooth bump function $\kappa$ 
with $[-1,1] \prec \kappa \prec [-2,2]$.
Let $\wt{N} = 2^{\lfloor \log_2 N \rfloor}$, 
and for every integer $0 \leq s \leq \lfloor \log_2 N \rfloor$ define
\begin{align}
\label{eq:mollif:phiDef}
	\phi^{(s)} \coloneqq
	\begin{cases}
	\kappa( 2^s N \,\cdot\, ) - \kappa( 2^{s+1} N \,\cdot\, )
	&	\text{if $1 \leq 2^s < \wt{N}$},
	\\
	\kappa( 2^s N \, \cdot \,)
	&	\text{if $2^s = \wt{N}$}.
	\end{cases}
\end{align}
Note that we have
$\Supp( \phi^{(s)} ) \subset \frac{1}{ 2^s N } I_s$,
where $I_s = \pm [\frac{1}{2},2]$
if $1 \leq 2^s < \wt{N}$, and $I_s = [-2,2]$
if $2^s = \wt{N}$.
Furthermore, for every dyadic integer $1 \leq Q \leq N$, we have
\begin{align}
\label{eq:mollif:PartUnity}
	\sum_{Q \leq 2^s \leq N} \phi^{(s)}
	= \kappa( Q N \,\cdot\,).
\end{align}

We let $N_1 = c_1 N$, for a small constant $c_1 \in (0,1]$.
It is easy to check that the intervals
$\frac{a}{q} + [ -\frac{2}{ Q N } , \frac{2}{Q N} ]$,
$1 \leq a \leq q$, $q \sim Q$, $1 \leq Q \leq N_1$ are all disjoint.
For a dyadic integer $Q$ and an integer $0 \leq s \leq \log_2 N$,
we define
\begin{align}
\label{eq:mollif:PhiDef}
	\Phi_{Q,s} = \sum_{\substack{ (a,q) = 1 \\ q \sim Q }}
	\tau_{-a/q} \phi^{(s)},
\end{align}
where $\tau_{-a/q} \phi^{(s)}(\alpha) := \phi(\alpha-a/q)$ is translation by $a/q$, so that
\begin{align}
\label{eq:mollif:PhiSupport}
	\Supp( \Phi_{Q,s} ) \subset
	\bigsqcup_{\substack{ (a,q) = 1 \\ q \sim Q }} 
	\bigg( \frac{a}{q} + \frac{I_s}{2^s N} \bigg)		
\end{align}
We also define the functions
\begin{align}
\label{eq:mollif:rhoDef}
	\lambda = \sum_{Q \leq N_1} \sum_{Q \leq 2^{s} \leq N} \Phi_{Q,s},
	\qquad
	\rho = 1 - \lambda.
\end{align}

\begin{proposition}
\label{thm:mollif:rhoSupport}
We have $0 \leq \lambda, \rho \leq 1$ and 
\begin{align*}
	\lambda = 1,\ \rho = 0
	\quad\text{on}\quad
	\bigsqcup_{Q \leq N_1} \bigsqcup_{\substack{ (a,q) = 1 \\ q \sim Q}} 
	\bigg( \frac{a}{q} + \bigg[ \! -\frac{1}{ Q N }, \frac{1}{ Q N } \bigg] \bigg).
\end{align*}
\end{proposition}

\begin{proof}
By~\eqref{eq:mollif:PartUnity}, we can rewrite $\lambda$ as
\begin{align*}
	\lambda = \sum_{Q \leq N_1} \sum_{\substack{(a,q) = 1 \\ q \sim Q}}
	\tau_{-a/q} \bigg( \sum_{Q \leq 2^s \leq N} \phi^{(s)} \bigg)
	= \sum_{Q \leq N_1} \sum_{\substack{(a,q) = 1 \\ q \sim Q}} \tau_{-a/q} \kappa( Q N \,\cdot\,).
\end{align*}
The proposition follows from the localization properties of $\kappa$.
\end{proof}

At this stage we define the fundamental domain 
$\calU = ( \frac{1}{2N_1} , 1 + \frac{1}{2N_1}]$,
and we note that when $N$ is large, 
then for every $1 \leq a \leq q \leq Q \leq N_1$,
the intervals $\frac{a}{q} + [ - \frac{2}{ Q N } , \frac{2}{ Q N } ]$ are contained in
$\overset{\circ}{\calU}$.
Therefore for $1 \leq Q \leq 2^s \leq N$, the functions 
$\phi^{(s)}$, $\Phi_{Q,s}$ and $\lambda$ are supported on the interior of $\calU$,
and they may be viewed as smooth functions over the torus $\T$,
by $1$-periodization from the interval $\calU$.
We will view $\Phi_{Q,s}$ alternatively as a smooth function
on the torus $\T$ or on the real line, but note that for an integer $n$,
$\wh{\Phi_{Q,s}}(n)$ has the same definition under both points of view.

For $n \in \Z$ and an integer $Q \geq 1$ we define a truncated divisor function
\begin{align*}
	d(n,Q) = \sum_{\substack{ \ 1 \leq d \leq Q \,: \\ d|n }} 1.
\end{align*}
The following useful lemma is due to Bourgain~\cite{Bourgain:ParabI}.

\begin{lemma}
\label{thm:mollif:DivBound}
Let $\delta_x$ be the Dirac function at $x$. Then
\begin{align*}
	&\phantom{(n \in \Z)} &
	\wh{ \sum_{\substack{ (a,q) = 1 \\ q \sim Q}} \delta_{a/q} }(n)
	\lesssim Q \cdot d(n,2Q)
	&&(n \in \Z).
\end{align*}
\end{lemma}


\begin{proposition}
We have
\begin{align}
	\label{eq:mollif:PhiAverage}
	&\phantom{(m \in \Z)} &
	\int \Phi_{Q,s} \dm
	&\lesssim \frac{Q^2}{ 2^s N },
	&&
	\\
	\label{eq:mollif:PhiFourierBound}
	&\phantom{(n \in \Z)} &
	\wh{\Phi_{Q,s}}(n) 
	&\lesssim \frac{Q}{ 2^s N } d(n,2Q)
	&&(n \in \Z)
\end{align}
\end{proposition}

\begin{proof}
Let $\gamma^{(s)} = \kappa - \kappa(2 \,\cdot\,)$ for $0 \leq s < \lfloor \log_2 N \rfloor$
and $\gamma^{(s)} = \kappa$ when $s = \lfloor \log_2 N \rfloor$.
By~\eqref{eq:mollif:phiDef} and~\eqref{eq:mollif:PhiDef}, 
we can write 
\begin{align*}
	\Phi_{Q,s}	
	= \sum_{\substack{(a,q)=1 \\ q \sim Q}} \tau_{-a/q} \gamma^{(s)}(2^s N \,\cdot\, )
	= \bigg( \sum_{\substack{(a,q)=1 \\ q \sim Q}} \delta_{a/q} \bigg) 
	\ast  \gamma^{(s)}(2^s N \,\cdot\, ).
\end{align*}
By Lemma~\ref{thm:mollif:DivBound}, we deduce the pointwise bound
\begin{align*}
	|\wh{\Phi_{Q,s}}(n)|
	=  \bigg| \wh{\sum_{\substack{(a,q)=1 \\ q \sim Q}} \delta_{a/q}}(n)
	\cdot \frac{1}{ 2^s N } \wh{\gamma^{(s)}}\Big( \frac{n}{2^s N} \Big)  \bigg|
	\lesssim \frac{Q}{ 2^s N } d(n,2Q),
\end{align*}
which is uniform in $n \in \Z$.
When $n = 0$ the left-hand side is $\int \Phi_{Q,s} \dm$.
\end{proof}

\begin{proposition}
For every $\eps > 0$ and $A > 0$, we have
\begin{align}
	\label{eq:mollif:rhoAverage}
	\int \rho \,\dm &\asymp 1,
	\\
	\label{eq:mollif:rhoFourierBound}
	\wh{\rho}(n) &\lesssim_{\eps,A} \frac{1}{N^{1-\eps}}
	\quad\text{for $0 < |n| \leq A N^A$.}
\end{align}
\end{proposition}

\begin{proof}
From~\eqref{eq:mollif:rhoDef} and~\eqref{eq:mollif:PhiAverage},
it follows that
\begin{align*}
	\int \rho \,\dm
	&= 1 - O\bigg(  \sum_{Q \leq N_1} \sum_{Q \leq 2^s \leq N} \frac{Q^2}{ 2^s N } \bigg)
	\\
	&= 1 - O\bigg( \frac{1}{N} \sum_{Q \leq N_1} Q \bigg)
	\\
	&= 1 - O\Big( \frac{N_1}{N} \Big).
\end{align*}
Since we have chosen $N_1 = c_0 N$ with $c_1$ small enough,
we have $\int \rho \dm \asymp 1$ as desired.
The bound on $\wh{\rho}$ is derived from~\eqref{eq:mollif:PhiFourierBound} 
in a similar fashion, using also the standard
divisor bound $d(n,Q) \leq d(n) \lesssim_\eps n^\eps$.
\end{proof}

\section{Restriction estimates}
\label{sec:qf}

We fix a non-degenerate integer quadratic form $R$ in $d$ variables.
In this section, 
we derive Theorem~\ref{thm:intro:TruncRestrQuad} 
from the introduction.
Note that the system of polynomials $(R(\bfx),\bfx)$
has total degree $d + 2$, hence the critical exponent in the
definite case is $p_d = \frac{2(d+2)}{d}$.
This is the exponent that arises in our argument,
even in the indefinite case, 
due to our use of $d$-dimensional exponential sum estimates
which do not depend on the type of quadratic form.
The larger critical exponent $p_{d,s} = \frac{2(d - s + 2)}{d - s}$
of Theorem~\ref{thm:intro:RestrHyperbParab} accounts
for the existence of a special linear subvariety of~\eqref{eq:prelims:QuadSurface},
but this does not influence our treatment of the truncated moment in 
Theorem~\ref{thm:intro:TruncRestrQuad}.

%
%

We use a smooth weight function $\omega : \R^d \rightarrow [0,1]$ of the form
\begin{align}
\label{eq:prelims:Weight}
	\omega = \eta\Big( \frac{\,\cdot\,}{N} \Big),
	\qquad
	\text{$\eta$ Schwarz function such that}\
	[-1,1]^d \prec \eta \prec [-2,2]^d.
\end{align}
Given a sequence $a : \Z^d \rightarrow \C$ 
supported on $[-N,N]^d$ with $\| a \|_2 = 1$
and a weight function $\omega : \Z^d \rightarrow [0,1]$
of the form~\eqref{eq:prelims:Weight},
we define
\begin{align}
	\label{eq:prelims:FaDef}
	F_a( \alpha, \bftheta )
	&= \sum_{\bfn \in \Z^d} a(\bfn) e( \alpha R(\bfn) + \bftheta \cdot \bfn )
	&&(\alpha \in \T, \bftheta \in \T^d),
	\\
	\label{eq:prelims:FDef}
	F(\alpha,\bftheta)
	&= \sum_{\bfn \in \Z^d} \omega(\bfn) e( \alpha R(\bfn) + \bftheta \cdot \bfn )
	&&(\alpha \in \T, \bftheta \in \T^d),
\end{align}
which are the extension operator of our surface $S$ acting on the sequence $a$ 
and the $\omega$-smoothed Fourier transform of the counting measure on $S$, 
respectively.

We will quote the estimates of Section~\ref{sec:mollif} extensively.
Via the Tomas-Stein argument in Section~\ref{sec:theorem}, we will devote most of our attention
to the complete exponential sum~\eqref{eq:prelims:FDef}.
The minor arc estimates of Appendix~\ref{sec:appqf}
yield the following in our context.

\begin{proposition}
\label{thm:qf:MinorArcBound}
Uniformly in $\alpha \in \T$, $\bftheta \in \T^d$, we have
\begin{align*}
	\rho(\alpha) \neq 0
	\quad\Rightarrow\quad	
	|F(\alpha,\bftheta)| &\lesssim N^{d/2}.
\end{align*}
\end{proposition}

\begin{proof}
We prove the contrapositive.
If $|F(\alpha,\bftheta)| \geq C_1 N^{d/2}$
for a large enough $C_1 > 0$, then by Proposition~\ref{thm:appqf:MinorArcBound}
there exist $a,q \in \Z$ such that $| \alpha - \frac{a}{q} | \leq \frac{c_1}{q N}$,
$1 \leq q \leq c_1 N$ and $(a,q) = 1$.
Consequently there exists a dyadic integer $Q$ such that $q \sim Q \Rightarrow Q \leq c_1 N = N_1$
and $|\alpha - \frac{a}{q}| \leq \frac{1}{QN}$,
so that $\rho(\alpha) = 0$ by Proposition~\ref{thm:mollif:rhoSupport}.
\end{proof}

For each dyadic integer $Q \geq 1$ and each integer $s \geq 0$
such that $1 \leq Q \leq 2^s$, we define a
piece of our original exponential sum by
\begin{align}
\label{eq:qf:FQsDef}
	F_{Q,s}(\alpha,\bftheta)
	= F(\alpha,\bftheta) \Big[ \Phi_{Q,s}(\alpha) - \frac{\int \Phi_{Q,s}}{\int \rho} \rho(\alpha) \Big].
\end{align}
We establish physical and Fourier bounds
for the exponential sums $F_{Q,s}$
via the major and minor arc estimates of Appendix~\ref{sec:appqf}.
It turns out to be important to have square-root cancellation
of the exponential sum $F$ on the minor arcs.
We also introduce a technical device to 
ensure that the Fourier transforms under consideration stay inside 
an $N^2 \times N \times \dots \times N$ box,
a fact that will prove useful later on.
Specifically, we fix a trigonometric polynomial $\psi_N$ on $\T^{d + 1}$ such that,
for a constant $C_R$ large enough with respect to $R$, 
\begin{align*}
	[-C_R N^2,C_R N^2] \times [-2N,2N]^d  \prec \wh{\psi}_N 
	\prec [-2 C_R N^2,2C_R N^2] \times [-4N,4N]^d,
\end{align*}
which in particular implies that $\int_{\T^{d + 1}} \psi_N = 1$.
When $H : \T^{d+1} \rightarrow \C$ is a bounded measurable function, 
we write $\dot{H} = H \ast \psi_N$ for brevity; 
note that $\| \dot{H} \|_p \leq \| H \|_p$ for any $p \geq 1$ by Young's inequality, 
and that $F = \dot{F}$ by Fourier inversion 
(since $\wh{F}$ is supported on the surface~\eqref{eq:prelims:QuadSurface}).

\begin{proposition}
\label{thm:qf:ArcPiecePhysFourierBounds}
Uniformly for $(m,\bfell) \in \Z^{d+1}$, we have 
\begin{align*}
	\| \dot{F}_{Q,s} \|_\infty 
	&\lesssim_\eps \bigg( \frac{2^s N}{Q} \bigg)^{\frac{d}{2}} Q^\eps,
	\\
	\wh{\dot{F}_{Q,s}}(m,\bfell)
	&\lesssim_\eps 
	1_{ |m| \lesssim N^2,\ |\bfell| \lesssim N }
	\Big( \frac{Q}{2^s N} d(m - R(\bfell),2Q) + \frac{Q^2}{ 2^s N^{2 - \eps} } \Big).
\end{align*}
\end{proposition}

\begin{proof}
When $\Phi_{Q,s}(\alpha) \neq 0$,
it follows from~\eqref{eq:mollif:PhiSupport}
that there exist $a,q \in \Z$
such that $q \sim Q$, $(a,q) = 1$ and
$|\alpha - \frac{a}{q}| \asymp \frac{1}{2^s N}$ if $2^s < \wt{N}$,
$|\alpha - \frac{a}{q}| \leq \frac{2}{2^s N}$ is $2^s = \wt{N}$.
By Propositions~\ref{thm:appqf:MajorArcBound} and~\ref{thm:qf:MinorArcBound},
and by~\eqref{eq:mollif:PhiAverage} and~\eqref{eq:mollif:rhoAverage}, it follows that,
uniformly in $\bftheta \in \R^d$,
\begin{align*}
	|F_{Q,s}(\alpha,\bftheta)| 
	&\lesssim_\eps Q^{-\frac{d}{2} + \eps} (2^s N)^{\frac{d}{2}} 
	+ \frac{Q^2}{2^s N} N^{\frac{d}{2}}
	\\
	&= \Big( \frac{2^s}{Q} \Big)^{\frac{d}{2}} Q^\eps N^\frac{d}{2}
	+ \frac{Q}{2^s} \cdot \frac{Q}{N} N^{\frac{d}{2}}
	\\
	&\leq \Big( \frac{2^s N}{Q} \Big)^\frac{d}{2} (1+Q^\eps).
\end{align*}
We let $\Psi_{Q,s} = \Phi_{Q,s} - \frac{\int \Phi_{Q,s}}{\int \rho} \rho$ and note that $\int \Psi
_{Q,s} = 0$ for each $Q,s$. 
Next we observe that, for any $(m,\bfell) \in \Z^{d+1}$,
\begin{align*}
	\wh{F_{Q,s}}(m,\bfell)
	&= \int_{\T^{d+1}} \Psi_{Q,s}(\alpha) F(\alpha,\bftheta) e( - \alpha m - \bftheta \cdot \bfell\, ) \dalpha \dbftheta
	\\
	&= \sum_{\bfn \in \Z^d} \omega(\bfn) 
	\int_{\T^{d+1}} \Psi_{Q,s}(\alpha) e\big( \alpha( R(\bfn) - m ) + \bftheta \cdot ( \bfn - \bfell ) \big) \dalpha \dbftheta
	\\
	&= \omega(\bfell) \wh{\Psi}_{Q,s}( m - R(\bfell) ).
\end{align*}
The second bound of the proposition then follows from the identity
\begin{align*}
	\wh{\dot{F}_{Q,s}}(m,\bfell)		
	= \psi_N(m,\bfell) \omega(\bfell) \wh{\Psi}_{Q,s}( m - R(\bfell) ) 1_{m \neq R(\bfell)},
\end{align*}
and the estimates~\eqref{eq:mollif:PhiAverage},~\eqref{eq:mollif:PhiFourierBound},
\eqref{eq:mollif:rhoAverage} and~\eqref{eq:mollif:rhoFourierBound}.
\end{proof}

We now define minor and major arc pieces of our exponential sum by
\begin{align}
\label{eq:qf:FDcp}
	F_\frakM = \sum_{Q \leq N_1} \sum_{Q \leq 2^s \leq N} F_{Q,s},
	\qquad
	F_\frakm = F - F_\frakM.
\end{align}
We can readily derive a uniform bound on the minor arc piece $F_\frakm$,
as an immediate consequence of 
the definition~\eqref{eq:mollif:rhoDef} and Proposition~\ref{thm:qf:MinorArcBound}.

\begin{proposition}
\label{thm:qf:MinorArcPieceBound}
We have $\| F_\frakm \|_\infty \lesssim N^{d/2}$.
\end{proposition}

The previous propositions also imply simple norm estimates for the operator of convolution with a major arc piece.

\begin{proposition}
\label{thm:qf:ConvolEasyBounds}
We have
\begin{align}
	\label{eq:qf:L1LinftyBound}
	\| \dot{F}_{Q,s} \ast f \|_\infty 
	&\lesssim \Big( \frac{2^s N}{Q} \Big)^\frac{d}{2} Q^\eps \| f \|_1,
	\\
	\label{eq:qf:L2L2EasyBound}
	\| \dot{F}_{Q,s} \ast f \|_2
	&\lesssim_\eps \frac{Q}{ 2^s N^{1-\eps} } \| f \|_2. 
\end{align}
\end{proposition}

\begin{proof}
Note that for any bounded function $W : \T^{d+1} \rightarrow \C$, we have
\begin{align*}
	\| W \ast f \|_\infty \leq \| W \|_\infty \| f \|_1,
	\qquad
	\| W \ast f \|_2 = \| \wh{W} \wh{f} \|_2 \leq \| \wh{W} \|_\infty \| f \|_2.
\end{align*}
It now suffices to apply these inequalities with $W = F_{Q,s}$ and
insert the estimates of 
Proposition~\ref{thm:qf:ArcPiecePhysFourierBounds}
(using also the bound $d(n,2Q) \lesssim n^\eps$).
\end{proof}

By interpolation, we can obtain an estimate for all moments.

\begin{proposition}
\label{thm:qf:InterpolEasyGenBound}
Let $p'_0 = \frac{2(d+2)}{d}$. For any $p' \in [2,\infty)$, we have
\begin{align}
\label{eq:qf:InterpolEasyGenBound}
	\| \dot{F}_{Q,s} \ast f \|_{p'}
	\lesssim \Big( \frac{2^s N}{Q} \Big)^{ (d+2) (\frac{1}{p'_0} - \frac{1}{p'}) } N^\eps \| f \|_p.
\end{align}
\end{proposition}

\begin{proof}
We interpolate between the estimates of Proposition~\ref{thm:qf:ConvolEasyBounds}
with $\theta \in (0,1)$ given by
\begin{align}
\label{eq:qf:thetaDef}
	\frac{1}{p'} = \frac{1 - \theta}{\infty} + \frac{\theta}{2},
	\qquad
	\frac{1}{p} = \frac{1 - \theta}{1} + \frac{\theta}{2}.
\end{align}
This yields
\begin{align*}
	\| \dot{F}_{Q,s} \ast f \|_{p'}
	&\lesssim
	\Big( \frac{2^s N}{Q} \Big)^{(1 - \theta) \frac{d}{2}}
	\cdot \Big( \frac{Q}{2^s N} \Big)^\theta \cdot N^\eps \cdot \| f \|_p
	\\
	&= \Big( \frac{2^s N}{Q} \Big)^{ \frac{d}{2} - ( 1 + \frac{d}{2} ) \theta } N^\eps \| f \|_p.
\end{align*}
Since $\theta = \frac{2}{p'}$, we may rewrite the exponent of $\frac{2^s N}{Q}$
as $(d + 2)(\frac{1}{p'_0} - \frac{1}{p'})$, which concludes the proof.
\end{proof}

At this stage we need a preparatory lemma on truncated divisor sums from \cite{Bourgain:ParabI}.

\begin{lemma}
\label{thm:qf:DivBound}
Let $D,Q,X \geq 1$ and $B \in \N$.
When $Q \leq 2X^{1/B}$, we have
\begin{align*}
	\#\{ |n| \leq X \,:\, d(n,Q) \geq D \} 
	\lesssim_{\eps,B} D^{-B} Q^\eps X.
\end{align*}
\end{lemma}

\begin{proof}
We show that
\begin{align*}
	\sum_{ |\ell| \leq X }	
	d( \ell, Q )^B
	\lesssim_{\eps,B} Q^\eps X,
\end{align*}
from which the result follows by Markov's inequality.
In the sum above,
the term $\ell = 0$ contributes at most $Q^B$,
and by~\cite[Eq.~(4.31)]{Bourgain:Squares} the other terms 
contribute at most $C_{\eps,B} Q^\eps X$.
The conclusion follows from our assumption on $Q$.
\end{proof}

We can now derive a more precise convolution bound 
using the previous lemma.

\begin{proposition}
\label{thm:qf:ConvolHardBound}
Let $B,D > 2$.
Uniformly for $Q \leq N^{2/B}$ and $Q \leq 2^s \leq N$, we have
\begin{align}
\label{eq:qf:L2L2HardBound}
	\| \dot{F}_{Q,s} \ast f \|_2 
	\lesssim_{\eps,B}
	\frac{D Q}{2^s N} \| f \|_2 
	+ \frac{ D^{ - \frac{B}{2} } Q^{1 + \eps} }{2^s N} N^{ \frac{d + 2}{2} } \| f \|_1.
\end{align}
\end{proposition}

\begin{proof}
By Parseval's identity and the bounds of Proposition~\ref{thm:qf:ArcPiecePhysFourierBounds}, we deduce that 
\begin{align*}
	\| \dot{F}_{Q,s} \ast f \|_2
	&= \Bigg[ \sum_{\substack{ |m| \lesssim N^2 \\ |\bfell| \lesssim N }} 
	|\wh{\dot{F}}_{Q,s}(m,\bfell)|^2 |\wh{f}(m,\bfell)|^2 \Bigg]^{1/2}
	\\
	&\lesssim \frac{Q}{2^s N} 
	\Bigg[ \sum_{\substack{ |m| \lesssim N^2 \\ |\bfell| \lesssim N }} 
	d( m - R(\bfell),2Q)^2 |\wh{f}(m,\bfell)|^2 \Bigg]^{1/2}
	+ \frac{Q^2}{2^s N^{2 - \eps}} \| \wh{f} \|_2
\end{align*}
We write $n = m - R(\bfell)$, so that assuming $Q \leq N^{2/B}$ and invoking 
Lemma~\ref{thm:qf:DivBound}, we obtain
\begin{align*}
	\| \dot{F}_{Q,s} \ast f \|_2 
	&\lesssim_{\eps,B} \frac{Q}{2^s N}
	\Big[ D^2 \| \wh{f} \|_2^2 + \| \wh{f} \|_\infty^2 N^d \times \#\{ |n| \lesssim N^2 \,:\, d(n,2Q) > D \} \Big]^{1/2}
	+ \frac{Q^2}{2^s N^{2 - \eps}} \| f \|_2
	\\
	&\lesssim \frac{Q}{2^s N} \Big( D^2 \| f \|_2^2 + D^{-B} Q^\eps N^{d + 2} \| f \|_1^2 \Big)^{1/2}
	+ \frac{Q}{2^s N} \cdot \frac{Q}{2^s N^{1 - \eps}} \| f \|_2.
\end{align*}
Since $B > 2$, we have that $Q \leq N^{1-\epsilon}$ for some $\epsilon>0$ and the last term may be absorbed into the first.
Finally we obtain
\begin{align*}
	\| \dot{F}_{Q,s} \ast f \|_2 
	\lesssim \frac{Q}{2^s N} 
	\big( D \| f \|_2 + Q^{\eps} D^{-\frac{B}{2}} N^{\frac{d+2}{2}} \| f \|_1 \big).
\end{align*}
\end{proof}

This new estimate can again be interpolated with
the $L^1 \rightarrow L^\infty$ one, to obtain the following bound.

\begin{proposition}
\label{thm:qf:InterpolHardBound}
Let $B,D > 2$.
Let $p'_0 = \frac{2(d+2)}{d}$ and $p' \in [2,\infty)$.
Uniformly for $Q \leq N^{2/B}$ and $Q \leq 2^s \leq N$, we have
\begin{align*}
	\| \dot{F}_{Q,s} \ast f \|_{p'}
	\lesssim_{\eps,B}
	D^{\frac{2}{p'}}
	\Big( \frac{2^s N}{Q} \Big)^{(d+2)(\frac{1}{p'_0} - \frac{1}{p'})} Q^\eps \| f \|_p
	+ D^{-\frac{B}{p'}} 	\Big( \frac{2^s N}{Q} \Big)^{(d+2)(\frac{1}{p'_0} - \frac{1}{p'})}
	 N^{\frac{d+2}{p'}}  Q^\eps \| f \|_1.
\end{align*}
\end{proposition}

\begin{proof}
Let $\theta \in (0,1]$ and $p' \geq 2$ 
be such that~\eqref{eq:qf:thetaDef} holds.
By convexity and~\eqref{eq:qf:L1LinftyBound}
and~\eqref{eq:qf:L2L2HardBound}, we have
\begin{align*}
	\| \dot{F}_{Q,s} \ast f \|_{p'} 
	&\leq \| \dot{F}_{Q,s} \ast f \|_{\infty}^{1 - \theta} \| \dot{F}_{Q,s} \ast f \|_2^\theta.
	\\
	&\lesssim_{\eps,B}
	Q^\eps \Big( \frac{2^s N}{Q} \Big)^{(1-\theta)\frac{d}{2}} 
	\cdot D^\theta \Big( \frac{Q}{2^s N} \Big)^\theta \cdot \| f \|_1^{1-\theta} \| f \|_2^\theta 
	\\
	&\phantom{\lesssim_{\eps,B} .} 
	+ Q^\eps \Big( \frac{2^s N}{Q} \Big)^{(1-\theta)\frac{d}{2}} 
	\cdot
	D^{-\theta \frac{B}{2}} \Big( \frac{Q}{2^s N} \Big)^{\theta}
	( N^{\frac{d+2}{2}} )^\theta \cdot \| f \|_1
\end{align*}
Since $|f|$ takes values in $\{0,1\}$, we may rewrite this as
\begin{align*}
	\| \dot{F}_{Q,s} \ast f \|_{p'}  
	&\lesssim_{\eps,B} 
	D^\theta \Big( \frac{2^s N}{Q} \Big)^{\frac{d}{2} - (1+\frac{d}{2}) \theta} Q^\eps \| f \|_p 
	+ D^{-\theta \frac{B}{2}} \Big( \frac{2^s N}{Q} \Big)^{\frac{d}{2} - (1+\frac{d}{2}) \theta}
	N^{\frac{d+2}{p'}} Q^\eps \| f \|_1.
\end{align*}
The proof is finished upon recalling that 
$\theta = \frac{2}{p'}$ by~\eqref{eq:qf:thetaDef}.
\end{proof}

We introduce a parameter $1 \leq Q_1 \leq N_1$ and write $F_\frakM = F_1 + F_2$ with
\begin{align}
\label{eq:qf:F1F2Def}
	F_1 = \sum_{Q \leq Q_1} \sum_{Q \leq 2^s \leq N} F_{Q,s},
	\qquad
	F_2 = \sum_{Q_1 < Q < N_1} \sum_{Q \leq 2^s \leq N} F_{Q,s}.
\end{align}

\begin{proposition}
\label{thm:qf:F1F2Bounds}
Suppose that $p' > p'_0 = \frac{2(d+2)}{d}$. 
Let $T \geq 1$ and suppose that $Q_1 \leq N^{2/B}$.
Then
\begin{align*}
	\| \dot{F}_1 \ast f \|_{p'}
	&\lesssim
	T^2 N^{ d - \frac{2(d+2)}{p'} } \| f \|_{p'}
	+ T^{-B} N^{d - \frac{d+2}{p'} } \| f \|_1,
	\\
	\| \dot{F}_2 \ast f \|_{p'}
	&\lesssim Q_1^{ - ( \frac{d}{2} - \frac{d+2}{p'} ) }
	N^{ d - \frac{2(d+2)}{p'} } \| f \|_p.
\end{align*}
\end{proposition}

\begin{proof}
By the triangle inequality
and Proposition~\ref{thm:qf:InterpolHardBound} 
with $T = D^{1/p'}$, it follows that
\begin{align*}
	\| \dot{F}_1 \ast f \|_{p'}
	&\lesssim \sum_{Q \leq Q_1} Q^{\eps - (d+2) ( \frac{1}{p'_0} - \frac{1}{p'} ) }
	\sum_{2^s \leq N} (2^s)^{ (d+2) ( \frac{1}{p'_0} - \frac{1}{p'} ) } 
	 N^{ (d+2) ( \frac{1}{p'_0} - \frac{1}{p'} ) } 
	\\
	&\phantom{\lesssim .}
	\cdot \big( T^2 \| f \|_p + T^{- B} N^{\frac{d+2}{p'}} \| f \|_1 \big).
	\\
	&\lesssim 
	T^2 N^{ 2(d+2) ( \frac{1}{p'_0} - \frac{1}{p'} ) } \| f \|_p
	+ T^{- B} N^{2(d+2) ( \frac{1}{p'_0} - \frac{1}{p'} ) - \frac{d+2}{p'}} \| f \|_1.
\end{align*}
It is easy to rewrite the exponents of $N$ in the desired form.

Turning our attention to $F_2$,
we deduce from the triangle inequality 
and~\eqref{eq:qf:InterpolEasyGenBound} that
\begin{align*}
	\| \dot{F}_2 \ast f \|_{p'} 
	&\lesssim 
	\sum_{Q > Q_1} Q^{- (d+2) ( \frac{1}{p'_0} - \frac{1}{p'} ) }	
	\sum_{2^s \leq N} (2^s)^{ (d+2) ( \frac{1}{p'_0} - \frac{1}{p'} ) } 
	\cdot N^{ (d+2) ( \frac{1}{p'_0} - \frac{1}{p'} ) }
	\cdot N^\eps \| f \|_p
	\\
	&\lesssim N^\eps Q_1^{ - ( \frac{d}{2} - \frac{d+2}{p'} ) }
	N^{ d - \frac{2(d+2)}{p'} } \| f \|_p.
\end{align*}
\end{proof}

\section{Proof of Theorem~\ref{thm:intro:TruncRestrQuad}}
\label{sec:theorem}

In this section we prove our theorem using the restriction estimates from Section~\ref{sec:qf} and Bourgain's~\cite{Bourgain:Squares,Bourgain:ParabI} 
discrete version of the Tomas--Stein argument~\cite[Chapter~7]{Wolff:Book}
from Euclidean harmonic analysis. 
We introduce a parameter $\lambda > 0$ and define
\begin{align}
\label{eq:prelims:TomasSteinNotation}
	E_\lambda = \{ |F_a| \geq \lambda \},
	\qquad
	f = 1_{E_\lambda} \frac{ F_a }{ |F_a| }.
\end{align}
Note that, by Cauchy-Schwarz in~\eqref{eq:prelims:FaDef},
we always have $|F_a| \leq CN^{d/2}$, and thus
we assume that the parameter $\lambda$ lies in $(0,CN^{d/2}]$.
Out theorem will quickly follow once we establish the following sharp level set bound. 

\begin{proposition}
\label{thm:qf:epsFreeLevelSetEst}
There exists $C>0$ such that,
for $\frac{2(d+2)}{d} < q \lesssim 1$,
\begin{align*}
	|E_\lambda| 
	\lesssim_q
	N^{\frac{dq}{2} - (d + 2)} \lambda^{-q}
	\qquad
	\text{for $\lambda \geq CN^{d/4}$}.	
\end{align*}
\end{proposition}

\begin{proof}
We view $a$ and $\omega$ as functions of $(R(\bfn),\bfn)$ for the sake of this argument,
so that $F = \wh{\omega 1_{S_{2N}}}$ and $F_a = \wh{a 1_{S_N}}$, where
\begin{align}
\label{eq:prelims:QuadSurface}
	S_{2N} = \{\, ( R( n_1 , \dots, n_d ), n_1 , \dots , n_d ) \,,\, n_i \in [-2N,2N] \cap \Z \,\}.
\end{align}
By Parseval, we have
\begin{align*}
	\lambda |E_\lambda|
	\leq \langle f , F_a \rangle_{L^2(\T^{d+1})}
	= \langle f , \wh{a 1_{S_N}} \rangle_{L^2(\T^{d+1})}
	= \langle \wh{f} , a \rangle_{\ell^2(S_N)}.
\end{align*}
By Cauchy-Schwarz and under the normalization $\| a \|_2 = 1$, it follows that
\begin{align*}
	\lambda^2 |E_\lambda|^2
	\leq \| f \|_{\ell^2(S_N)}^2
	\leq \langle f \cdot \omega 1_{S_{2N}} , f \rangle_{\ell^2(\Z^{d+1})}.
\end{align*}
By another application of Parseval, we conclude that
\begin{align}
\label{eq:prelims:TomasStein}
	\lambda^2 |E_\lambda|^2 
	\leq \langle f \ast F, f \rangle_{L^2(\T^{d+1})}
\end{align}
We will use this inequality to obtain bounds
of the expected order on the level sets $E_\lambda$.
By our earlier observation $F = \dot{F}$, inequality \eqref{eq:prelims:TomasStein} becomes 
\begin{align*}
	\lambda^2 |E_\lambda|^2 \leq |\langle \dot{F} \ast f , f \rangle|,
\end{align*}
and recalling the decompositions~\eqref{eq:qf:FDcp} 
and~\eqref{eq:qf:F1F2Def}, we have
\begin{align*}
	\lambda^2 |E_\lambda|^2
	&\leq |\langle \dot{F}_{\frakm} \ast f , f \rangle|
	+ |\langle \dot{F}_2 \ast f , f \rangle|
	+ |\langle \dot{F}_1 \ast f , f \rangle|
	\\
	&\leq \| F_\frakm \|_\infty \| f \|_1^2 +
	\| \dot{F}_2 \ast f \|_{p'} \| f \|_p
	+ \| \dot{F}_1 \ast f \|_{p'} \| f \|_p.
\end{align*}

Let $T \geq 1$ be a parameter to be determined later,
and assume that we have chosen $Q_1$
so that $Q_1 \leq N^{2/B}$.
Inserting the estimates of
Propositions~\ref{thm:qf:MinorArcPieceBound}
and~\ref{thm:qf:F1F2Bounds}, this yields
\begin{align*}
	\lambda^2 |E_\lambda|^2
	&\lesssim N^{d/2} |E_\lambda|^2
	+ N^\eps Q_1^{ - ( \frac{d}{2} - \frac{d+2}{p'} ) } N^{ d - \frac{2(d+2)}{p'} } \| f \|_p^2
	\\
	&\phantom{\lesssim .} 
	+ T^2 N^{ d - \frac{2(d+2)}{p'} } \| f \|_p^2
	+ T^{- B} N^{d - \frac{d+2}{p'}} \| f \|_p \| f \|_1.
\end{align*} 
Assume that $\lambda \geq CN^{d/4}$ for $C > 0$ large enough
and fix $Q_1 = N^{\eps_1}$, where $\eps_1 = 1/B$.
For $p' > 2(d+2)/d$, 
and provided that $\eps$ is small enough,
we have then
\begin{align*}
	\lambda^2 |E_\lambda|^2 
	\lesssim T^2 N^{ d - \frac{2(d+2)}{p'} } |E_\lambda|^{2 - \frac{2}{p'}}
	+ T^{-B} N^{ d - \frac{d+2}{p'} } |E_\lambda|^{2 - \frac{1}{p'}}.
\end{align*}
Writing $\lambda = \eta N^{d/2}$ with
$\eta \in (0,1]$, we have
therefore either
\begin{align*}
	|E_\lambda|^{\frac{2}{p'}}
	\lesssim T^2 N^{ - \frac{2(d+2)}{p'}} \eta^{-2}
	\quad\text{or}\quad
	|E_\lambda|^{\frac{1}{p'}}
	\lesssim T^{-B} N^{ - \frac{d+2}{p'}} \eta^{-2}.
\end{align*}
Write $D = T^{p'}$, so that in either case
\begin{align*}
	|E_\lambda|
	\lesssim D N^{- (d+2) } \eta^{-p'}
	+ D^{-B} N^{- (d+2) } \eta^{-2p'}.
\end{align*}
Choose $D = \eta^{-\nu}$ for a parameter $\nu> 0$, so that
\begin{align*}
	|E_\lambda| \lesssim N^{- (d+2) } \eta^{- p'-\nu}( 1 + \eta^{ - p' + (B+1) \nu } ).
\end{align*}
Choosing $B \geq C/\nu$ with $C > 0$ large enough,
we deduce that $|E_\lambda| \lesssim N^{-(d+2)} \eta^{- (p' + \nu)}$.
Since $q \coloneqq p' + \nu$ can be chosen arbitrarily close to $\frac{2(d+2)}{d}$,
this finishes the proof,
upon recalling that $\eta = \lambda N^{-d/2}$.
\end{proof}

\textit{Proof of Theorem~\ref{thm:intro:TruncRestrQuad}.}
We may certainly assume that $\|a\|_2 = 1$ in proving this result.
We apply Proposition~\ref{thm:qf:epsFreeLevelSetEst} 
for a certain $\frac{2(d+2)}{d} < q < p$ to obtain
\begin{align*}
	\int_{ |F_a| \geq CN^{d/4} }
	|F_a|^p \dm
	&= p \int_{CN^{d/4}}^{N^{d/2}}
	\lambda^{p-1} |E_\lambda| \dlambda
	\\
	&\lesssim_p
	N^{\frac{dq}{2} - (d+2)} \int_1^{N^{d/2}} \lambda^{p - q - 1} \dlambda.
	\\
	&\lesssim_p
	N^{\frac{dp}{2} - (d+2)}.
\end{align*}
\qed

\appendix

\section{Bounds on quadratic exponential sums}
\label{sec:appqf}

In this appendix we derive 
standard major and minor arc bounds
on exponential sums associated to quadratic forms,
which we could not locate precisely in the literature.
We fix a nondegenerate quadratic form $R$ 
in $d$ variables with integer matrix,
and we define 
\begin{align*}
	&\phantom{( \alpha \in \T, \bftheta \in \T^d )} &
	F_R(\alpha,\bftheta)
	&= \sum_{\bfn} \omega(\bfn) e( \alpha R(\bfn) + \bftheta \cdot \bfn )
	&&( \alpha \in \T, \bftheta \in \T^d ).
\end{align*}
Our first minor-arc-type bound
is obtained by the standard Weyl differentiation
process for forms of high dimension
(see~\cite[Section~8.3.1.1]{Magyar:Birch}
or Davenport~\cite[Chapter~13]{Davenport:Book}).

\begin{proposition}
\label{thm:appqf:MinorArcBound}
Let $d \geq 1$.
For every $c_0 \in (0,1]$, there exists
a constant $C > 0$ depending at most on $c_0,d,R$ 
such that the following holds.
If $|F_R(\alpha,\bftheta)| \geq CN^{d/2}$,
there exist $a,q \in \Z$ such that
$|\alpha - \frac{a}{q}| \leq \frac{c_0}{qN}$,
$1 \leq q \leq c_0 N$ and $(a,q) = 1$.
\end{proposition}

\begin{proof}
By definition, we have $R(\bfx) = \bfx^\transp M \bfx$,
where $M$ is a symmetric, non-singular integer $d \times d$ matrix.
For a vector $\bfx \in \R^d$,
we write $|\bfx| = \max( |x_1|,\dots,|x_n| )$
and $\| x \| = \min_{\bfn \in \Z^d} |\bfx - \bfn|$.
By squaring, we have
\begin{align*}
	|F_R( \alpha, \bftheta )|^2
	= \sum_{\bfn,\bfm \in \Z^d} \omega(\bfm) \omega(\bfn) e\big( \alpha (R(\bfm) - R(\bfn)) + \bftheta \cdot (\bfm - \bfn) \big).
\end{align*}
Letting $\bfm = \bfn + \bfu$ and $\Delta^{\times}_{\bfu} \omega(\bfn) = \omega(\bfn) \omega( \bfn + \bfu)$, we deduce that
\begin{align*}
	|F_R( \alpha, \bftheta )|^2
	&= \sum_{ |\bfu| \leq 4N } e( \alpha R(\bfu) + \bftheta \cdot \bfu ) 
	\sum_{\bfn \in \Z^d} \Delta^{\times}_{\bfu} \omega(\bfn) e( \bfn \cdot ( 2\alpha M \bfu ) )
	\\
	&\leq \sum_{ |\bfu| \leq 4N } | \wh{ \Delta^{\times}_{\bfu} \omega }( 2 \alpha M \bfu ) |.
\end{align*}
Since $\Delta^{\times}_{\bfu} \omega = \eta( \frac{\cdot}{N} ) \eta( \frac{\,\cdot\, + \bfu}{N})$ has support in $[-2N,2N]^d$
and satisfies $\| \partial^{\bfalpha} \Delta^{\times}_{\bfu} \omega \|_\infty \lesssim N^{-|\bfalpha|}$ 
for all $\bfalpha \in (\N \cup \{ 0 \})^d$,
one can verify through an application of Poisson's formula that 
$|\wh{ \Delta^{\times}_{\bfu} \omega }(\bfxi)| \lesssim_A N^d ( 1 + N \| \bfxi \| )^{-A}$
uniformly for $\bfxi \in \R^d$, for any $A > 0$.
Therefore
\begin{align*}
	|F_R( \alpha, \bftheta )|^2
	&\lesssim_d N^d \sum_{ |\bfu| \leq 4N } \big( 1 + N \| 2 \alpha M \bfu \| \big)^{-(d+1)}
	\\
	&\lesssim N^d \sum_{ |\bfr| \leq \frac{N}{2} } ( 1 + |\bfr| )^{-(d+1)} \cdot
	\#\big\lbrace |\bfu| \leq 4N \,:\, 2 \alpha M \bfu \in \tfrac{\bfr}{N} + [-\tfrac{1}{2N},\tfrac{1}{2N}] \bmod 1 \big\rbrace.
\end{align*}
If $\bfu$, $\bfu'$ belong to the set above,
then $\| 2 \alpha M (\bfu - \bfu') \| \leq \frac{1}{N}$, 
and therefore
\begin{align*}
	|F_R( \alpha, \bftheta )|^2
	\lesssim  N^d \cdot
	\#\{ |\bfu| \leq 8N \,:\, \| 2 \alpha M \bfu \| \leq \tfrac{1}{N} \}.
\end{align*}
Let $1 \leq L \leq N$ be a new parameter.
By a similar reasoning, if we partition the box $[-8N,8N]^d$
into subboxes of sidelength at most $L$,
and if we partition the box $[-\frac{1}{N},\frac{1}{N}]^d$
into subboxes of sidelength at most $\frac{L}{N^2}$, we obtain
\begin{align*}
	|F_R( \alpha, \bftheta )|^2
	\lesssim_d  L^{-2d} N^{3d} \cdot
	\#\{ |\bfu| \leq L \,:\, \| 2 \alpha M \bfu \| \leq \tfrac{4L}{N^2} \}.
\end{align*}
We choose $L = c_1 N$, where $c_1 \in (0,1]$ is to be determined later. 
If $|F_R( \alpha, \bftheta )| \geq C N^{d/2}$ for a large enough constant $C$
(depending on $d$ and $c_1$),
then there exists $\bfu \neq 0$ such that $|\bfu| \leq c_1 N$ and
$\| 2\alpha M \bfu \| \leq \frac{c_1}{qN}$,
and we let $q = 2|M \bfu|$.
Since $M$ is non-singular, we have $1 \leq q \lesssim_M c_1 N$
and $\| q \alpha \| \leq \frac{4c_1}{N}$,
and therefore there exists $a \in \Z$
such that $| \alpha - a/q | \leq \frac{4c_1}{qN}$.
This finishes the proof upon choosing $c_1$ small enough 
with respect to $d$, $M$ and upon reducing $a$ and $q$.
\end{proof}

On the major arcs, we use a standard majorant obtained through
the Poisson formula, using the square-root level of cancellation in 
the Gaussian sum and oscillatory integral associated
to non-degenerate quadratic forms.

\begin{proposition}
\label{thm:appqf:MajorArcBound}
Let $d \geq 1$.
Suppose that $\alpha \in \R$ is of the form $\alpha = \frac{a}{q} + \beta$ with
$a,q \in \Z$, $\beta\in\R$ such that 
$\|\beta\| \lesssim \frac{1}{qN}$, $1 \leq q \lesssim N$ and $(a,q) = 1$.
Then
\begin{align*}
	|F_R(\alpha,\bftheta)| \lesssim_\eps q^{ - d/2 + \eps } \min( N^d, |\beta|^{- d/2} ).
\end{align*}
\end{proposition}

\begin{proof}
We define a Gaussian sum and an oscillatory integral by
\begin{align*}
	S(a,\bfb;q) = \sum_{\bfu \in \Z_q^d} e_q( a R(\bfu) + \bfb \cdot \bfu ),
	\qquad
	I(\beta,\bfgamma;N) = \int_{\R^d} \eta(\bfx) e( \beta N^2 R(\bfx) + N \bfgamma \cdot \bfx ) \dbfx.
\end{align*}
We write $\alpha \equiv \frac{a}{q} + \beta \bmod 1$
and we sum over residue classes modulo $q$ to obtain
\begin{align*}
	F_R(\alpha,\bftheta)
	= \sum_{\bfu \in \Z_q^d} e_q( a R(\bfu) ) 
	\sum_{\substack{\bfn \in \Z^d \,: \\ \bfn \equiv \bfu \bmod q }} \omega(\bfn) e( \beta R(\bfn) + \bftheta \cdot \bfn ).
\end{align*}
Writing $1_{\bfn \equiv \bfu \bmod q} = q^{-d} \sum_{\bfb \in \Z_q^d} e_q( \bfb \cdot (\bfu - \bfn) )$, we arrive at
\begin{align*}
	F_R( \alpha, \bftheta )
	= \sum_{\bfb \in \Z_q^d} q^{-d} S( a , \bfb ; q )
	\sum_{\bfn \in \Z^d} \omega(\bfn) e( \beta R(\bfn) + ( \bftheta - \tfrac{\bfb}{q} ) \cdot \bfn ).  
\end{align*}
By Poisson's formula and rescaling, it follows that
\begin{align}
	F_R( \alpha,\bftheta )
	\label{eq:appweyl:PoissonDcp}
	= \sum_{\bfb \in \Z_q^d} q^{-d} S( a , \bfb ; q ) 
	\sum_{ \bfm \in \Z^d } N^d \cdot I( \beta, \bftheta - \tfrac{\bfb}{q} - \bfm ; N ).
\end{align}
We write $I( \beta, \bftheta - \tfrac{\bfb}{q} - \bfm ; N ) 
= \int_\R \eta(\bfx) e( N \phi_{\bfb,\bfm}(\bfx) ) \dbfx$,
where
\begin{align*}
	\phi_{\bfb,\bfm}(\bfx) = \beta N R(\bfx) + ( \bftheta - \tfrac{\bfb}{q} - \bfm ) \cdot \bfx.
\end{align*}
On the support of $\eta$, we have $|x| \leq 2$ and therefore
\begin{align*}
	\nabla \phi_{\bfb,\bfm}(\bfx) = \theta - \tfrac{\bfb}{q} - \bfm + O( \tfrac{1}{q} )
\end{align*}
under our size condition on $\beta$.
We fix a large enough constant $C > 0$.

For $|\bfm| \geq C$, we have $|\nabla \phi_{\bfb,\bfm}| \asymp |\bfm|$ on $\Supp \eta$,
and therefore by stationary phase~\cite[Chapter VII]{Stein:Book}
we have $| \int_\R \eta  e( N \phi_{\bfb,\bfm}  ) |  \lesssim (N|\bfm|)^{-(d+1)}$.
For $\| \bftheta - \frac{\bfb}{q} \| \geq \frac{C}{q}$,
we have $|\nabla \phi_{\bfb,\bfm}| \asymp | \bftheta - \frac{\bfb}{q} - \bfm | \gtrsim \| \bftheta - \frac{\bfb}{q} \|$
on $\Supp \eta$ and $\| \frac{\phi_{\bfb,\bfm}}{ |\bftheta - \frac{\bfb}{q} - \bfm| } \|_{C^2} \lesssim 1$,
so that by stationary phase again we deduce that
$| \int_\R \eta  e( N \phi_{\bfb,\bfm}  ) | \lesssim ( N \| \bftheta - \frac{\bfb}{q} \| )^{-d}$.
Finally, for $|\bfm| \leq C$ and $ \| \bftheta - \frac{\bfb}{q} \| \leq \frac{C}{q}$,
we note that the phase is a non-degenerate quadratic form
and therefore we have an oscillatory integral estimate~\cite[Section~6]{Wolff:Book} of the form
$| \int_\R \eta  e( N \phi_{\bfb,\bfm}  ) | \lesssim (1 + |\beta| N^2)^{-d/2}$.

For the Gaussian sum, we use 
the simple squaring-differencing bound
$|S(a,\bfb;q)| \lesssim q^{d/2}$ for $(a,q) = 1$
(see e.g.~\cite{IK:Book} Lemma~20.12).
Inserting these various estimates into~\eqref{eq:appweyl:PoissonDcp} yields
\begin{align*}
	|F_R(\alpha,\bftheta)|
	&\lesssim_\eps
	q^{ - d/2 } \sum_{\substack{ \| \bftheta - \frac{\bfb}{q} \| \leq \frac{C}{q} \\ |\bfm| \leq C }} N^d (1 + |\beta| N^2)^{-\frac{d}{2}}
	\\
	&\phantom{\lesssim_\eps .}
	+ q^{ - d/2 } \sum_{\substack{ \| \bftheta - \frac{\bfb}{q} \| \geq \frac{C}{q} \\ |\bfm| \leq C }} \| \bftheta - \tfrac{\bfb}{q} \|^{-d}
	\ +\ q^{ d/2 } \sum_{|\bfm| \geq C} N^{-1} |\bfm|^{-(d+1)}
	\\
	&\lesssim 
	q^{-d/2 + \eps} N^d (1 + |\beta| N^2)^{-d/2} \,+\, q^{d/2 + \eps}.
\end{align*}
The second term may be absorbed into the first since
$|\beta| \lesssim \frac{1}{qN}$
and $1 \leq q \lesssim N$, and this concludes the proof.
\end{proof}

\section{A diagonalization argument}
\label{sec:appdiag}

In this section we present a simple argument,
possibly well-known to experts, by which
Theorem~\ref{thm:intro:RestrHyperbParab}
follows from~\cite[Corollary~1.3]{BD:Hypersurf}.

\textit{Proof of Theorem~\ref{thm:intro:RestrHyperbParab}.}

Let $Q$ be a non-singular quadratic form with integer coefficients.
Fix a sequence $a : \Z^d \rightarrow \C$ supported on $[-N,N]^d$;
by homogeneity we may assume $\| a \|_2 = 1$.
We let
\begin{align*}
	I
	= \| F_a \|_p^p
	= \int_{[-\frac{1}{2},\frac{1}{2}]^{d+1}} \bigg| \sum_{\bfn \in \Z^d} a(\bfn) e\big( \alpha Q(\bfn) + \bftheta \cdot \bfn \big) \bigg|^p \dalpha \dbftheta.	
\end{align*}
We pick a linear transformation $T$ of $\Q^d$ such that $Q = D \circ T$,
where $D$ is a diagonal form with coefficients $\pm 1$.
Then by defining the lattice $\Lambda = T(\Z^d)$ and 
by a change of variables $\bftheta = T^*(\bfxi)$, we have
\begin{align*}
	I 
	&= \int_{[-\frac{1}{2},\frac{1}{2}]^{d+1}} \bigg| \sum_{\bfm \in \Lambda} a(T^{-1}(\bfm)) 
	e\big( \alpha D(\bfm) + (T^{-1})^*(\bftheta) \cdot \bfm \big) \bigg|^p \dalpha \dbftheta,
	\\
	&= |\det T| \int_{E} \bigg| \sum_{\bfm \in \Lambda} a(T^{-1}(\bfm)) e\big( \alpha D(\bfm) + \bfxi \cdot \bfm \big) \bigg|^p \dalpha \dbfxi,
\end{align*}
where $E = [-\frac{1}{2},\frac{1}{2}] \times (T^*)^{-1}( [-\frac{1}{2},\frac{1}{2}]^d )$.
We have $\Lambda \subset q^{-1} \Z^d$ for some $q \in \N$ depending on $Q$,
and by a change of variables $(\alpha,\bfxi) \leftarrow (q^2 \alpha,q\bfxi)$,
we have
\begin{align*}
	I =
	q^{d+2} |\det T| \int_{F} 
	\bigg| \sum_{\bfell \in q \Lambda} a(T^{-1}(\bfell/q)) e\big( \alpha D(\bfell) +  \bfxi \cdot \bfell \big) \bigg|^p \dalpha \dbfxi,
\end{align*}
where $F = [-\frac{1}{2q^2},\frac{1}{2q^2}] \times (T^*)^{-1}( [-\frac{1}{2q},\frac{1}{2q}]^d )$.
Finally, we can cover $F$ by finitely many translated copies
of $[-\frac{1}{2},\frac{1}{2}]^{d+1}$,
and since $q \Lambda \subset \Z^d \cap [-CN,CN]^d$,
we may apply the usual restriction estimate for diagonal forms of 
Bourgain-Demeter~\cite[Corollary~1.3]{BD:Hypersurf} to obtain the estimate
\begin{align*}
	\| F_a \|_p^p
	\lesssim_\eps
	\begin{cases}
		N^{\frac{sp}{2} - s + \eps} \| a \|_2^p 
		& \text{for $2 \leq p \leq \frac{2(d - s + 2)}{d - s}$,}
		\\
		N^{\frac{dp}{2} - (d+2) + \eps} \| a \|_2^p
		&\text{for $p > \frac{2(d - s + 2)}{d-s}$.}
	\end{cases}
\end{align*}
The $N^\eps$ factor in the supercritical range can be removed 
via (a minor modification of) Bourgain's $\eps$-removal lemma
for the paraboloid $( x_1 , \dots , x_d , x_1^2 + \dotsb + x_d^2 )$.
(Alternatively, one can use Theorem~\ref{thm:intro:TruncRestrQuad}
to remove this factor, via the standard $\eps$-removal process~\cite[Lemma~3.1]{HH:epsremoval_parab}).
\qed

\bibliographystyle{amsplain}
\bibliography{epsremoval_quad}

\bigskip

\textsc{\footnotesize Department of mathematics,
University of British Columbia,
Room 121, 1984 Mathematics Road,
Vancouver BC V6T 1Z2, Canada
}

\textit{\small Email address: }\texttt{\small khenriot@math.ubc.ca}

\textsc{\footnotesize Heilbronn Institute for Mathematical Research
School of Mathematics,
University of Bristol,
Howard House, Queens Avenue, 
Bristol BS8 1SN, United Kingdom
}

\textit{\small Email address: }\texttt{\small kevin.hughes@bristol.ac.uk}

\end{document}